\documentclass[12pt]{amsart}
\usepackage{hyperref,cleveref,amssymb}

\theoremstyle{plain}
\newtheorem{theorem}{Theorem}[section]

\theoremstyle{definition}

\newtheorem{example}[theorem]{Example}

\newcommand{\abs}[1]{\lvert#1\rvert}
\newcommand{\norm}[1]{\lVert#1\rVert}
\newcommand{\bigabs}[1]{\bigl\lvert#1\bigr\rvert}
\newcommand{\bignorm}[1]{\bigl\lVert#1\bigr\rVert}
\newcommand{\Bigabs}[1]{\Bigl\lvert#1\Bigr\rvert}
\newcommand{\Bignorm}[1]{\Bigl\lVert#1\Bigr\rVert}
\newcommand{\term}[1]{{\textit{\textbf{#1}}}}   
\renewcommand{\mid}{\::\:}

\DeclareSymbolFont{bbold}{U}{bbold}{m}{n}
\DeclareSymbolFontAlphabet{\mathbbold}{bbold}

\def\bx{\boldsymbol{x}}
\def\bs{\boldsymbol{s}}
\def\bt{\boldsymbol{t}}
\def\ba{\boldsymbol{a}}

\DeclareMathOperator{\Range}{Range}

\renewcommand{\le}{\leqslant}
\renewcommand{\ge}{\geqslant}

\begin{document}

\title[Krivine's Function Calculus and Bochner integration]
{Krivine's Function Calculus\\ and Bochner integration}

\author{V.G. Troitsky}
\address{Department of Mathematical and Statistical Sciences,
         University of Alberta, Edmonton, AB, T6G\,2G1, Canada.}
\email{troitsky@ualberta.ca}

\author{M.S. T\"urer}
\address{Department of Mathematics and Computer Science, \.Istanbul
K\"ult\"ur University, Bak\i rk\"oy 34156, \.Istanbul, Turkey}
\email{m.turer@iku.edu.tr}

\thanks{The first author was supported by an NSERC grant.}



\date{\today}

\begin{abstract}
  We prove that Krivine's Function Calculus is compatible with
  integration. Let $(\Omega,\Sigma,\mu)$ be a finite measure space,
  $X$ a Banach lattice, $\bx\in X^n$, and
  $f\colon\mathbb R^n\times\Omega\to\mathbb R$ a function such that
  $f(\cdot,\omega)$ is continuous and positively homogeneous for every
  $\omega\in\Omega$, and $f(\bs,\cdot)$ is integrable for every
  $\bs\in\mathbb R^n$. Put $F(\bs)=\int f(\bs,\omega)d\mu(\omega)$ and
  define $F(\bx)$ and $f(\bx,\omega)$ via Krivine's Function
  Calculus. We prove that under certain natural assumptions
  $F(\bx)=\int f(\bx,\omega)d\mu(\omega)$, where the right hand side
  is a Bochner integral.
\end{abstract}

\maketitle

\section{Motivation}

In \cite{Kalton:12}, the author defines a real-valued
function of two real or complex variable via
\begin{math}
  F(s,t)=\int_0^{2\pi}\bigabs{s+e^{i\theta}t}d\theta.
\end{math}
This is a positively homogeneous continuous function. Therefore, given two
vectors $u$ and $v$ in a Banach lattice~$X$, one may apply Krivine's
Function Calculus to $F$ and consider $F(u,v)$ as an element of~$X$.
The author then claims that
\begin{equation}
  \label{Kalton}
  F(u,v)=\int_0^{2\pi}\bigabs{u+e^{i\theta}v}d\theta,
\end{equation}
where the right hand side here is understood as a Bochner integral;
this is used later in \cite{Kalton:12} to conclude that
\begin{math}
  \bignorm{F(u,v)}\le
   \int_0^{2\pi}\bignorm{u+e^{i\theta}v}d\theta
\end{math}
because Bochner integrals have this property:
\begin{math}
  \bignorm{\int f}\le\int\norm{f}.
\end{math}
A similar exposition is also found
in~\cite[p.~146]{Davis:84}. Unfortunately, neither \cite{Kalton:12}
nor \cite{Davis:84} includes a proof of~\eqref{Kalton}.  In this note,
we prove a general theorem which implies~\eqref{Kalton} as a special
case. 

\section{Preliminaries}

We start by reviewing the construction of Krivine's Function Calculus
on Banach lattices; see~\cite[Theorem~1.d.1]{Lindenstrauss:79} for
details. For Banach lattice terminology, we refer the reader
to~\cite{Abramovich:02,Aliprantis:06a}. 

Fix $n\in\mathbb N$. A function
$F\colon\mathbb R^n\to\mathbb R$ is said to be \term{positively
  homogeneous} if
\begin{displaymath}
  F(\lambda t_1,\dots,\lambda t_n)
  =\lambda F(t_1,\dots,t_n)
  \mbox{ for all }t_1,\dots,t_n\in\mathbb R\mbox{ and }\lambda\ge 0. 
\end{displaymath}
Let $H_n$ be the set of all continuous positively homogeneous
functions from $\mathbb R^n$ to~$\mathbb R$. Let $S_\infty^n$ be the
unit sphere of $\ell_\infty^n$, that is,
\begin{displaymath}
  S_\infty^n=\bigl\{(t_1,\dots,t_n)\in\mathbb R^n\mid
  \max\limits_{i=1,\dots,n}\abs{t_i}=1\bigr\}.
\end{displaymath}
It can be easily verified that the restriction map $F\mapsto
F_{|S_\infty^n}$ is a lattice isomorphism from $H_n$ onto
$C(S_\infty^n)$. Hence, we can identify $H_n$ with $C(S_\infty^n)$. For each
$i=1,\dots,n$, the $i$-th coordinate projection
$\pi_i\colon\mathbb R^n\to\mathbb R$ clearly belongs to~$H_n$. 

Let $X$ be a (real) Banach lattice and $\bx=(x_1,\dots,x_n)\in
X^n$. Let $e\in X_+$ be such that $x_1,\dots,x_n$ belong to $I_e$, the
principal order ideal of $e$. For example, one could take
$e=\abs{x_1}\vee\dots\vee\abs{x_n}$. By Kakutani's representation
theorem, the ideal $I_e$ equipped with the norm
\begin{displaymath}
    \norm{x}_e=\inf\bigl\{\lambda>0\mid\abs{x}\le\lambda e\bigr\}
\end{displaymath}
is lattice isometric to $C(K)$ for some compact Hausdorff $K$.  Let
$F\in H_n$. Interpreting $x_1,\dots,x_n$ as elements of $C(K)$, we can
define $F(x_1,\dots,x_n)$ in $C(K)$ as a composition. We may view it
as an element of $I_e$ and, therefore, of $X$; we also denote it by
$\widetilde F$ or $\Phi(F)$. It may be shown that, as an element of
$X$, it does not depend on the particular choice of $e$. This results
in a (unique) lattice homomorphism $\Phi\colon H_n\to X$ such that
$\Phi(\pi_i)=x_i$. The map $\Phi$ will be referred to as
\term{Krivine's function calculus}. This construction allows one to
define expressions like
$\Bigl(\sum_{i=1}^n\abs{x_i}^p\Bigr)^{\frac1p}$ for $0<p<\infty$ in
every Banach lattice $X$; this expression is understood as $\Phi(F)$
where
$F(t_1,\dots,t_n)=\Bigl(\sum_{i=1}^n\abs{t_i}^p\Bigr)^{\frac1p}$. Furthermore,
\begin{equation}
  \label{FC-norms}
  \bignorm{F(\bx)}\le
  \norm{F}_{C(S_\infty^n)}\cdot\Bignorm{\bigvee_{i=1}^n\abs{x_i}}.
\end{equation}

Let $L_n$ be the sublattice of $H_n$ or, equivalently, of
$C(S_\infty^n)$, generated by the coordinate projections
$\pi_i$ as $i=1,\dots,n$. It follows from the Stone-Weierstrass
Theorem that $L_n$ is dense in $C(S_\infty^n)$. It follows from
$\Phi(\pi_i)=x_i$ that $\Phi(L_n)$ is the sublattice generated by
$x_1,\dots,x_n$ in $X$, hence $\Range\Phi$ is contained in
the closed sublattice of $X$ generated by $x_1,\dots,x_n$. It follows
from, e.g., Exercise~8 on \cite[p.204]{Aliprantis:06a} that this
sublattice is separable.

\medskip

Let $(\Omega,\Sigma,\mu)$ be a finite measure space and $X$ a Banach
space. A function $f\colon\Omega\to X$ is \emph{measurable} if there
is a sequence $(f_n)$ of simple functions from $\Omega$ to $X$ such
that $\lim_n\norm{f_n(\omega)-f(\omega)}=0$ almost everywhere. If, in
addition, $\int\norm{f_n(\omega)-f(\omega)}d\mu(\omega)\to 0$ then $f$ is
\term{Bochner integrable} with $\int_Af\,d\mu=\lim_n\int_Af_n\,d\mu$
for every measurable set~$A$. In the following theorem, we collect a
few standard facts about Bochner integral for future reference; we
refer the reader to~\cite[Chapter II]{Diestel:77} for proofs and
further details.

\begin{theorem}\label{Boch}
  Let $f\colon\Omega\to X$.
  \begin{enumerate}
  \item\label{Boch-lim} If $f$ is the almost everywhere limit of a
    sequence of measurable functions then $f$ is measurable.
  \item\label{Boch-norming} If $f$ is separable-valued and there is a
    norming set $\Gamma\subseteq X^*$ such that $x^*f$ is measurable
    for every $x^*\in\Gamma$ then $f$ is measurable.
  \item\label{Boch-int} A measurable function $f$ is Bochner integrable
    iff $\norm{f}$ is integrable.
  \item\label{Boch-L1} If $f(\omega)=u(\omega)x$ for some fixed
    $x\in X$ and $u\in L_1(\mu)$ and for all $\omega$ then $f$ is
    measurable and Bochner integrable.
  \item\label{Boch-op} If $f$ is Bochner integrable and
    $T\colon X\to Y$ is a bounded operator from $X$ to a Banach space
    $Y$ then
  \begin{math}
  T\bigl(\int f\,d\mu\bigr)=
  \int Tf\,d\mu.
  \end{math}
  \end{enumerate}
\end{theorem}

\section{Main theorem}

Throughout the rest of the paper, we assume that $(\Omega,\Sigma,\mu)$
is a finite measure space, $n\in\mathbb N$, and
$f\colon\mathbb R^n\times\Omega\to\mathbb R$ is such that
$f(\cdot,\omega)$ is in $H_n$ for every $\omega\in\Omega$ and
$f(\bs,\cdot)$ is integrable for every $\bs\in\mathbb R^n$. For every
$\bs\in\mathbb R^n$, put $F(\bs)=\int f(\bs,\omega)d\mu(\omega)$. It
is clear that $F$ is positively homogeneous.

Suppose, in addition, that $F$ is continuous. Let $X$ be a Banach
lattice, $\bx\in X^n$, and $\Phi\colon H_n\to X$ the corresponding
function calculus. Since $F\in H_n$, $\widetilde{F}=F(\bx)=\Phi(F)$ is
defined as an element of~$X$.  On the other hand, for every~$\omega$,
the function $\bs\in\mathbb R^n\mapsto f(\bs,\omega)$ is in~$H_n$,
hence we may apply $\Phi$ to it. We denote the resulting vector by
$\tilde{f}(\omega)$ or $f(\bx,\omega)$. This produces a function
$\omega\in\Omega\mapsto f(\bx,\omega)\in X$.

\begin{theorem}\label{CK}
  Suppose that $F$ is continuous and the function
  $M(\omega):=\bignorm{f(\cdot,\omega)}_{C(S_\infty^n)}$ is
  integrable. Then $f(\bx,\omega)$ is Bochner integrable as a function
  of $\omega$ and
  \begin{math}
    F(\bx)=\int f(\bx,\omega)d\mu(\omega),
  \end{math}
  where the right hand side is a Bochner integral.
\end{theorem}

\begin{proof}
  \emph{Special case:} $X=C(K)$ for some compact Hausdorff $K$. By
  uniqueness of function calculus, Krivine's function calculus $\Phi$
  agrees with ``point-wise'' function calculus. In particular,
  \begin{displaymath}
    \widetilde{F}(k)=F\bigl(x_1(k),\dots,x_n(k)\bigr)\mbox{ and }
    \bigl(\tilde f(\omega))(k)=f\bigl(x_1(k),\dots,x_n(k),\omega\bigr)
  \end{displaymath}
  for all $k\in K$ and $\omega\in\Omega$. We view
  $\tilde f$ as a function from $\Omega$ to $C(K)$.

  We are going to show that $\tilde f$ is Bochner integrable.  It
  follows from $\tilde f(\omega)\in\Range\Phi$ that $\tilde f$ a
  separable-valued function. For every $k\in K$, consider the
  point-evaluation functional $\varphi_k\in C(K)^*$ given by
  $\varphi_k(x)=x(k)$. Then
  \begin{displaymath}
    \varphi_k\bigl(\tilde f(\omega)\bigr)=\bigl(\tilde f(\omega))(k)
    =f\bigl(x_1(k),\dots,x_n(k),\omega\bigr).
  \end{displaymath}
  for every $k\in K$. By assumptions, this function is integrable; in
  particular, it is measurable. Since the set
  \begin{math}
    \bigl\{\varphi_k\mid k\in K\bigr\}
  \end{math}
  is norming in $C(K)^*$, Theorem~\ref{Boch}\eqref{Boch-norming}
  yields that $\tilde f$ is measurable.

  Clearly,
  \begin{math}
    \bigabs{\bigl(\tilde f(\omega)\bigr)(k)}\le M(\omega)
  \end{math}
  for every $k\in K$ and $\omega\in\Omega$, so that
  $\norm{\tilde f(\omega)}_{C(K)}\le M(\omega)$ for every $\omega$. It
  follows that
  \begin{math}
    \int\norm{\tilde f(\omega)}_{C(K)}\, d\mu(\omega)
  \end{math}
  exists and, therefore, $\tilde f$ is Bochner integrable by
  Theorem~\ref{Boch}\eqref{Boch-int}.

  Put $h:=\int\tilde f(\omega)\,d\mu(\omega)$, where the right hand side is
  a Bochner integral.  Applying Theorem~\ref{Boch}\eqref{Boch-op}, we
  get
  \begin{multline*}
    h(k)=\varphi_k(h)
    =\int\varphi_k\bigl(\tilde f(\omega)\bigr)\,d\mu(\omega)
    =\int f\bigl(x_1(k),\dots,x_n(k),\omega\bigr)\,d\mu(\omega)\\
    =F\bigl(x_1(k),\dots,x_n(k)\bigr)=\widetilde{F}(k).
  \end{multline*}
  for every $k\in K$. It follows that $\int\tilde
  f(\omega)\,d\omega=\widetilde{F}$.

  \emph{General case.} Let $e=\abs{x_1}\vee\dots\abs{x_n}$. Then
  $\bigl(I_e,\norm{\cdot}_e\bigr)$ is lattice isometric to $C(K)$ for
  some compact Hausdorff $K$. Note also that $\abs{x}\le\norm{x}_ee$
  for every $x\in I_e$; this yields $\norm{x}\le\norm{x}_e\norm{e}$,
  hence the inclusion map
  $T\colon \bigl(I_e,\norm{\cdot}_e\bigr)\to X$ is
  bounded. Identifying $I_e$ with $C(K)$, we may view $T$ as a bounded
  lattice embedding from $C(K)$ into $X$.

  By the construction on Krivine's Function Calculus, $\Phi$ actually
  acts into $I_e$, i.e., $\Phi=T\Phi_0$, where $\Phi_0$ is the
  $C(K)$-valued function calculus. By the special case, we know that
  $\int\tilde f(\omega)\,d\mu(\omega)=\widetilde{F}$ in
  $C(K)$. Applying $T$, we obtain the same identity in $X$ by
  Theorem~\ref{Boch}\eqref{Boch-op}.
\end{proof}

Finally, we analyze whether any of the assumptions may be
removed. Clearly, one cannot remove the assumption that $F$ is
continuous; otherwise, $\widetilde{F}$ would make no sense. The
following example shows that, in general, $F$ need not be continuous.

\begin{example}
  Let $n=2$, let $\mu$ be a measure on $\mathbb N$ given by
  $\mu\bigl(\{k\}\bigr)=2^{-k}$. For each $k$, we define
  $f_k=f(\cdot,k)$ as follows. Note that it suffices to
  define $f_k$ on $S^2_\infty$. Let $I_k$ be the straight
  line segment connecting $(1,0)$ and
  $(1,2^{-k+1})$. Define $f_k$ so that it vanishes
  on $S^2_\infty\setminus I_k$,
  $f_k(1,0)=f_k(1,2^{-k+1})=0$,
  $f_k(1,2^{-k})=2^k$, and is linear on each half of
  $I_k$. Then $f_k\in H_2$ and $F(\bs)$ is defined for every
  $\bs\in\mathbb R^2$. It follows from
  $F(\bs)=\sum_{k=1}^\infty 2^{-k}f_k(\bs)$ that $F(1,0)=0$ and
  $F(1,2^{-k})\ge 2^{-k}f_k(1,2^{-k})=1$, hence $F$ is discontinuous at
  $(1,0)$.
\end{example}

The assumption that $M$ is integrable cannot be removed
as well. Indeed, consider the special case when $X=C(S_\infty^n)$ and
$x_i=\pi_i$ as $i=1,\dots,n$. In this case, $\Phi$ is the identity map
and $\tilde f(\omega)=f(\cdot,\omega)$. It follows from
Theorem~\ref{Boch}\eqref{Boch-int} that $\tilde{f}$ is Bochner
integrable iff $\norm{\tilde{f}}$ is integrable iff $M$ is integrable.

Finally, the assumption that $f(\cdot,\omega)$ is in $H_n$ for every
$\omega$ may clearly be relaxed to ``for almost every $\omega$''.

\section{Direct proof}

In the previous section, we presented a proof of Theorem~\ref{CK}
using representation theory. In this section, we present a direct
proof. However, we impose an additional assumption: we assume
that $f(\cdot,\omega)$ is continuous on $S_\infty^n$ uniformly
on~$\omega$, that is,
\begin{multline}
  \label{ucont}
  \mbox{for every }\varepsilon>0\mbox{ there exists }\delta>0
  \mbox{ such that }
  \bigabs{f(\bs,\omega)-f(\bt,\omega)}<\varepsilon\\
  \mbox{ for all }\bs,\bt\in S_\infty^n
  \mbox{ and all }\omega\in\Omega
  \mbox{ provided that }
  \norm{\bs-\bt}_\infty<\delta.  
\end{multline}

In Theorem~\ref{CK}, we assumed that $F$ was continuous and $M$ was
integrable. Now these two conditions are satisfied automatically.  In
order to see that $F$ to is continuous, fix $\varepsilon>0$; let
$\delta$ be as in~\eqref{ucont}, then
\begin{equation}
  \label{FsFt}
  \bigabs{F(\bs)-F(\bt)}
  \le\int\bigabs{f(\bs,\omega)-f(\bt,\omega)}
  d\mu(\omega)<\varepsilon\mu(\Omega)
\end{equation}
whenever $\bs,\bt\in S_\infty^n$ with $\norm{\bs-\bt}_\infty<\delta$.
The proof of integrability of $M$ will be included in the proof of the
theorem.

\begin{theorem}\label{direct}
  Suppose that $f(\cdot,\omega)$ is continuous on $S_\infty^n$
  uniformly on~$\omega$. Then $f(\bx,\omega)$ is Bochner integrable as
  a function of $\omega$ and
  \begin{math}
    F(\bx)=\int f(\bx,\omega)d\mu(\omega).
  \end{math}
\end{theorem}

\begin{proof}
  Without
loss of generality, by scaling $\mu$ and~$\bx$, we may assume that
$\mu$ is a probability measure and
\begin{math}
  \Bignorm{\bigvee_{i=1}^n\abs{x_i}}=1;
\end{math}
this will simplify computations.  In particular,~\eqref{FC-norms} becomes
$\norm{H(\bx)}\le\norm{H}_{C(S_\infty^n)}$ for every
$H\in C(S_\infty^n)$. Note also that $\bx$ in the theorem is a
``fake'' variable as $\bx$ is fixed. It may be more accurate to write
$\widetilde{F}$ and $\tilde{f}(\omega)$ instead of $F(\bx)$ and
$f(\bx,\omega)$, respectively. Hence, we need to prove that $\tilde f$
as a function from $\Omega$ to $X$ is Bochner integrable and its Bochner
integral is~$\widetilde{F}$.

Fix $\varepsilon>0$. Let $\delta$ be as in~\eqref{ucont}.
It follows from~\eqref{FsFt} that
\begin{equation}
  \label{FsFteps}
   \bigabs{F(\bs)-F(\bt)}<\varepsilon
   \mbox{ whenever }\bs,\bt\in S_\infty^n
   \mbox{ with }\norm{\bs-\bt}_\infty<\delta.
\end{equation}
Each of the $2n$ faces of $S_\infty^n$ is a translate of the
$(n-1)$-dimensional unit cube $B_\infty^{n-1}$. Partition each of
these faces into $(n-1)$-dimensional cubes of diameter less
than~$\delta$, where the diameter is computed with respect to the
$\norm{\cdot}_\infty$-metric. Partition each of these cubes into
simplices. Therefore, there exists a partition of the entire
$S_\infty^n$ into finitely many simplices of diameter less
than~$\delta$. Denote the vertices of these simplices by
$\bs_1,\dots,\boldsymbol{s_m}$. Thus, we have produced a
triangularization of $S_\infty^n$ with nodes
$\bs_1,\dots,\boldsymbol{s_m}$.

Let $\ba\in\mathbb R^m$. Define a function $L\colon
S_\infty^n\to\mathbb R$ by setting $L(\boldsymbol{s_j})=a_j$ as
$j=1,\dots,m$ and then extending it to each of the simplices linearly;
this can be done because every point in a simplex can be written in a
unique way as a convex combination of the vertices of the simplex. We
write $L=T\ba$. This gives rise to a linear operator
$T\colon\mathbb R^m\to C(S_\infty^n)$. For each $j=1,\dots,m$, let
$e_j$ be the $j$-th unit vector in $\mathbb R^m$; put
$d_j=Te_j$. Clearly,
\begin{equation}
  \label{Ta}
  T\ba=\sum_{j=1}^ma_jd_j\mbox{ for every }\ba\in\mathbb R^m. 
\end{equation}

Let $H\in C(S_\infty^n)$. Let $L=T\ba$ where
$a_j=H(\boldsymbol{s_j})$. Then $L$ agrees with $H$ at
$\bs_1,\dots,\bs_m$. We write $L=SH$; this
defines a linear operator $S\colon C(S_\infty^n)\to
C(S_\infty^n)$. Clearly, this is a linear contraction.

Suppose that $H\in C(S_\infty^n)$ is such that
$\bigabs{H(\bs)-H(\bt)}<\varepsilon$ whenever
$\norm{\bs-\bt}_\infty<\delta$. Let $L=SH$. We claim that
$\bignorm{L-H}_{C(S_\infty^n)}<\varepsilon$. Indeed, fix
$\bs\in S_\infty^n$. Let $\bs_{j_1},\dots,\bs_{j_n}$ be the
vertices of a simplex in the triangularization of $S_\infty^n$ that
contains~$\bs$. Then $\bs$ can be written as a convex combination
$\bs=\sum_{k=1}^n\lambda_k\bs_{j_k}$. Note that
$\norm{\bs-\bs_{j_k}}_\infty<\delta$ for all $j=1,\dots,n$. It
follows that
\begin{displaymath}
  \bigabs{L(\bs)-H(\bs)}
  =\Bigabs{\sum_{k=1}^n\lambda_kL(\bs_{j_k})
    -\sum_{k=1}^n\lambda_kH(\bs)}
    \le\sum_{j=1}^n\lambda_k\bigabs{H(\bs_{j_k})-H(\bs)}
    <\varepsilon.
\end{displaymath}
This proves the claim.


Let $G=SF$. It follows from~\eqref{FsFteps} and the preceding observation 
$\norm{G-F}_{C(S_\infty^n)}<\varepsilon$, so that
\begin{equation}\label{GF}
  \bignorm{G(\bx)-F(\bx)}<\varepsilon.
\end{equation}
Similarly, for every $\omega\in\Omega$, apply $S$ to $f(\cdot,\omega)$
and denote the resulting function $g(\cdot,\omega)$. In particular,
$g(\bs_j,\omega)=f(\bs_j,\omega)$ for every $\omega\in\Omega$ and
every $j=1,\dots,m$. It follows also that
\begin{equation}
  \label{fgseps}
  \bignorm{f(\cdot,\omega)-g(\cdot,\omega)}_{C(S_\infty^n)}<\varepsilon
\end{equation}
for every~$\omega$, and, therefore
\begin{equation}
  \label{fgxeps}
  \bignorm{\tilde{f}(\omega)-\tilde{g}(\omega)}
  =\bignorm{f(\bx,\omega)-g(\bx,\omega)}<\varepsilon,
\end{equation}
where $\tilde{g}(\omega)=g(\bx,\omega)$ is the image under $\Phi$ of
the function $\bs\in S_\infty^n\mapsto g(\bs,\omega)$. Note that
\begin{equation}
  \label{Ggs}
  G(\bs_j)=F(\bs_j)=\int f(\bs_j,\omega)d\mu(\omega)
  =\int g(\bs_j,\omega)d\mu(\omega)
\end{equation}
for every $j=1,\dots,m$. Since $G=SF=T\ba$ where
$a_j=F(\bs_j)=G(\bs_j)$ as $j=1,\dots,m$, it follows from~\eqref{Ta}
that
\begin{equation}
  \label{Gsum}
  G=\sum_{j=1}^mG(\bs_j)d_j.
\end{equation}
Similarly, for every
$\omega\in\Omega$, we have
\begin{equation}
  \label{gsum}
  g(\cdot,\omega)=\sum_{j=1}^mg(\bs_j,\omega)d_j.
\end{equation}
Applying $\Phi$ to~\eqref{Gsum}
and~\eqref{gsum}, we obtain
$\widetilde{G}=G(\bx)=\sum_{j=1}G(\bs_j)d_j(\bx)$ and
\begin{displaymath}
  \tilde{g}(\omega)=g(\bx,\omega)
  =\sum_{j=1}^mg(\bs_j,\omega)d_j(\bx)
  =\sum_{j=1}^mf(\bs_j,\omega)d_j(\bx).
\end{displaymath}
Together with Theorem~\ref{Boch}\eqref{Boch-L1},
this yields that $\tilde{g}$ is measurable and Bochner
integrable.  It now follows from~\eqref{Ggs} and~\eqref{Gsum} that
\begin{multline}
  \label{Gg}
  G(\bx)=\sum_{j=1}G(\bs_j)d_j(\bx)
  =\sum_{j=1}^m\Bigl(\int g(\bs_j,\omega)d\mu(\omega)\Bigr)d_j(\bx)\\
  =\int\Bigl(\sum_{j=1}^mg(\bs_j,\omega)d_j(\bx)\Bigr)d\mu(\omega)
  =\int g(\bx,\omega)d\mu(\omega).
\end{multline}

We will show next that $\tilde{f}$ is Bochner integrable. It follows
from~\eqref{fgxeps} and the fact that $\varepsilon$ is arbitrary that
$\tilde{f}$ can be approximated almost everywhere (actually,
everywhere) by measurable functions; hence $\tilde{f}$ is measurable
by Theorem~\ref{Boch}\eqref{Boch-lim}. Next, we claim that there
exists $\lambda\in\mathbb R_+$ such that
$\bigabs{f(\bs,\omega)-f(\boldsymbol{1},\omega)}\le\lambda$ for all
$\bs\in S_\infty^n$ and all $\omega\in\Omega$. Here
$\boldsymbol{1}=(1,\dots,1)$. Indeed, let $\bs\in S_\infty^n$ and
$\omega\in\Omega$. Find $j_1,\dots,j_l$ such that
$\bs_{j_1}=\boldsymbol{1}$, $\bs_{j_k}$ and $\bs_{j_{k+1}}$ belong to
the same simplex as $k=1,\dots,l-1$, and $\bs_{j_l}$ is a vertex of a
simplex containing~$\bs$. It follows that
\begin{displaymath}
  \bigabs{f(\bs,\omega)-f(\boldsymbol{1},\omega)}
    \le\bigabs{f(\bs,\omega)-f(\bs_{j_l},\omega)}
    +\sum_{k=1}^{l-1}\bigabs{f(\bs_{j_{k+1}},\omega)-f(\bs_{j_k},\omega)}
    \le l\varepsilon\le m\varepsilon.
\end{displaymath}
This proves the claim with $\lambda=m\varepsilon$. It follows that
\begin{displaymath}
  \bignorm{\tilde{f}(\omega)}
  \le\bignorm{f(\cdot,\omega)}_{C(S_\infty^n)}
  =\sup_{\bs\in S_\infty^n}\bigabs{f(\bs,\omega)}
  \le\bigabs{f(\boldsymbol{1},\omega)}+\lambda.
\end{displaymath}
Since $\bigabs{f(\boldsymbol{1},\omega)}+\lambda$ is an integrable function
of~$\omega$, we conclude that $\norm{\tilde{f}}$ is integrable, hence
$\tilde{f}$ is Bochner integrable by
Theorem~\ref{Boch}\eqref{Boch-int}. It now follows from~\eqref{fgxeps}
that
\begin{equation}
  \label{fg}
  \Bignorm{\int f(\bx,\omega)d\mu(\omega)
           -\int g(\bx,\omega)d\mu(\omega)}
  \le\int\bignorm{f(\bx,\omega)-g(\bx,\omega)}d\mu(\omega)
  <\varepsilon
\end{equation}

Finally, combining~\eqref{GF}, \eqref{Gg}, and~\eqref{fg}, we get
\begin{displaymath}
  \Bignorm{F(\bx)-\int f(\bx,\omega)d\mu(\omega)}< 2\varepsilon.
\end{displaymath}
Since $\varepsilon>0$ is arbitrary, this proves the theorem.
\end{proof}

Some of the work on this paper was done during a visit of the second
author to the University of Alberta. We would like to thank the
referee whose helpful remarks and suggestions considerably improved
this paper.

\end{document}